\documentclass{amsart}
\usepackage{color,amsthm}

\theoremstyle{theorem}
\newtheorem{defn}{Definition}[section]
\newtheorem{thm}[defn]{Theorem}

\newtheorem{lem}[defn]{Lemma}

\newtheorem{prop}[defn]{Proposition}
\newtheorem{cor}[defn]{Corollary}

\def\B{{\mathcal B}}

\def\C{{\mathcal C}}

\def\G{{\mathcal G}}
\def\H{{\mathcal H}}

\def\DS{{\mathcal DS}}

\def\kphi{{ K}_{\varphi}}

\def\diam{{\text{diam}}}
\def\N{{\mathbb N}}
\def\R{{\mathbb R}}

\def\2n{\{0,1\}^{\N}}

\begin{document}

\title[Visible and Invisible Cantor sets]{Visible and Invisible Cantor sets}
\author[C.~Cabrelli]{Carlos Cabrelli}
\address[C. Cabrelli and U. Molter]{Departamento de Matem\'atica\\
Facultad de Ciencias Exactas y Naturales\\
Universidad de Buenos Aires\\
Pabell\'on I, Ciudad Universitaria\\
C1428EGA C.A.B.A.\\
Argentina \\
and IMAS - CONICET, Argentina}
\email[C.Cabrelli]{cabrelli@dm.uba.ar}
\urladdr{http://mate.dm.uba.ar/~cabrelli}
\email[U.~Molter]{umolter@dm.uba.ar}
\urladdr{http://mate.dm.uba.ar/~umolter}
\thanks{C. Cabrelli and U. Molter are partially supported by Grants UBACyT
X638 and X502 (UBA),  and PIP 112-200801-00398 (CONICET)}
\author[U.~B.~Darji]{Udayan B. Darji}
\address[U. ~B.~Darji]{Department of Mathematics\\
University of Louisville\\
Louisville, KY 40292 }
\email[U.~Darji]{ubdarj01@louisville.edu}
\thanks{U. Darji is partially supported by University of Louisville Project Initiation Grant.}
\author[U.Molter]{Ursula M. Molter}
\thanks{This paper is in final form and no version of it will be submitted
for publication elsewhere.}

\begin{abstract}
In this article we study for which Cantor sets there exists a {\em gauge}-function $h$, such that the $h$-Hausdorff-measure is positive and finite. 
We show that the collection of sets for which this is true is dense in the set of all compact subsets of a Polish space $X$.
More general, any {\em generic} Cantor set satisfies that there exists a translation-invariant measure $\mu$ for which the set has positive and finite $\mu$-measure.

In contrast, we generalize an example of Davies of dimensionless Cantor sets (i.e. a Cantor set for which any translation invariant measure is either $0$ or non-$\sigma$-finite) that enables us to show that the collection of these sets is also dense in the set of all compact subsets of a Polish space $X$.
\end{abstract}

\maketitle

\section{Introduction}
Measure theoretic dimension theory provides a fundamental tool to classify sets.
However Hausdorff dimension, as well as other notions of dimension such as
packing dimension and Minkowski dimension are not completely satisfactory.

For example there are many examples of compact sets whose Hausdorff 
measure at its critical exponent is zero or infinite.
This could also happen even if we consider  the generalized Hausdorff $h$-measure were $h$ is a appropriate  gauge function in a well defined class. See for example \cite{B}.
Furthermore, in 1930 Davies \cite{D} produced a beautiful example of a Cantor set on the real line
whose $\mu$-measure is zero or no $\sigma$-finite for every  translation invariant Borel measure $\mu$.
See \cite{MT} for more results in this direction.

In this article we study this phenomena. 
We try to estimate in some way the size of the class of  visible sets, that is sets that have positive 
and finite measure for some $h$-Hausdorff measure or some translation invariant Borel  measure.

We focus on Cantor sets in the context of Polish spaces.
In \cite{CMMS04} the authors proved that a large class of Cantor sets defined 
by monotone gap-sequences are visible.
They explicitly construct the corresponding gauge function  $h$.
Here we extend this result to a larger class in  general Polish spaces (Theorem~\ref{visibility}). 
 We also obtain density results for the class of visible sets and study {\it generic} visibility for subsets of the real line.
 
 Then we focus on the concept of strong invisibility (see def \ref{invisible}). We were able to extend the ideas in the construction of Davies
 to a general abelian Polish group, obtaining a big class of strongly invisible compact sets in these groups.
 We also prove that the set of strongly invisible sets in the space of compact sets in the line with the Hausdorff distance is dense.

 The paper is organized as follows. We first introduce some notation and terminology in section \ref{sec-2}. 
 A key ingredient will be the definition of visibility and strong-invisibility and the analysis of the appropriate topology to be able to state density results.
 In section~\ref{sec-visible} we show that a large class of Cantor sets is visible, and in section  \ref{sec-invisible} we show how to construct many strongly invisible sets.

\section{Terminology and Notation}\label{sec-2} Throughout $X$ will denote a {\em Polish space}, i.e., a separable space with a complete metric.  We
let $\C(X)$ denote the set of all compact subsets of $X$ endowed with the Hausdroff metric $d_H$.
We recall that for $X$ Polish, $\C(X)$ is Polish and for $X$ compact, $\C(X)$ is compact. We let $\B(X)$ denote the set of Borel subsets of $X$. 

A subset of a Polish space
is a {\em Cantor space} (or {\em Cantor set}) if it is compact, has no isolated points and has a basis of clopen sets, i.e.,
sets which are simultaneously open and closed.  There is always an homeomorphism between two Cantor sets.
We also consider several special types of Cantor sets subsets of the reals. We note that for a subset
of the reals to be a Cantor set, it suffices to have the properties of being compact, perfect and containing no interval.

We now describe a general way of describing any Cantor set subset of $[0,1]$ of Lebesgue measure zero which contains $\{0,1\}$. Let $D$ be the set of dyadic rationals in $(0,1)$. 

$$D = \left \{ i2^{-k}| 1 \le i \le 2^k-1, k \in \N \right \},$$ 

$$\G =\left \{\varphi:D \rightarrow (0,1) |\sum_{d \in D} \varphi (d) = 1  \right \},$$
and for each $\varphi \in \G$ we associate the function,
$$ \Phi (d) = \sum_{d' \in D, d' < d} \varphi(d').$$

The function $\varphi$ can be though of as a density function supported on $D$ and $\Phi$ is the associated cumulative distribution function.
Associated to $\varphi$, we define the Cantor set $K _{\varphi}$ as follows: 

$$\kphi = [0,1] \setminus \bigcup_{d \in D} \left (\Phi (d), \Phi (d) + \varphi (d) \right).$$
We think of $\varphi$ as the ``gap function" of $\kphi$. We have the following basic facts.

\begin{prop}\label{basic1}
For each $\varphi \in \G$, $\kphi$ is Cantor set subset of $[0,1]$ containing $\{0,1\}$
with Lebesgue measure zero. Conversely, given any Cantor set $K$ subset of $[0,1]$ with
Lebesgue measure zero which also contains $\{0,1\}$, there is $\varphi \in \G$ such that $K =\kphi$.\end{prop}
\begin{prop}\label{uniquephi} Suppose $\varphi_1, \varphi_2 \in \G$ are such that $K _{\varphi_1} = K _{\varphi_2}$. Then,
there is a homeomorphism $g$ on $[0,1]$, mapping $D$ onto itself such that $\varphi_2 = \varphi_1 \circ g$. \end{prop}

The following special subclass of $\kphi$'s was studied in \cite{CMMS04}. Let 
$$\DS = \left \{ \alpha \in (\R {^+}) ^{\N}: \alpha \mbox{ is decreasing and } \sum_{n=1}^{\infty} \alpha(i) =1  \right \} .$$ 
For each $\alpha \in \DS$, we define $K_{\alpha} = K_{\varphi}$
where  $\varphi(1/2) =\alpha (1) $, $\varphi(1/4)=\alpha (2) $, 
$\varphi(3/4) =\alpha (3), \dots, \varphi(\frac{2s+1}{2^j})=\alpha (2^{j-1}+s) $. We call the sequence $\alpha$  a "gap sequence". 

We introduce necessary terminology and notation concerning measures. 
Let 
$$\H = \{h:[0,\infty) \rightarrow [0, \infty) | h (0) =0, h \mbox{ is continuous and nondecreasing} \}, $$
and $\mu^h$ be the associated Hausdorff measure defined on the Borel subsets of $X$. 
\begin{defn} We call $M \in \B (X)$
{\bf $\H$-visible} if there is $h \in \H$ such that $0 < \mu _h (M) < \infty$, i.e., $M$ is an $h$-set for some $h \in \H$.  
\end{defn}

In \cite{CMMS04} it was shown that for any $\alpha \in \DS$, $K_{\alpha}$ is $\H$-visible.
 
The main purpose of this article is determine wether the previous result can be extended to other $K_\varphi$, or if not, how big is the class of Cantor sets for which this is true.

In Polish groups, i.e., topological groups with Polish topology, Hausdorff measures are particular instances of general translation invariant Borel measures. 
\begin{defn} \label{invisible} Let $X$ be a Polish group. 
A set $M \in \B (X)$ is
called {\bf visible} if there exists a translation invariant Borel measure $\mu$ on $\B (X)$ such that $0 < \mu (M) < \infty$. A set $M \in \B (X)$ is
called {\bf strongly invisible} if for every translation invariant Borel measure $\mu$ on $\B (X)$ we have that $\mu (M) =0$ or $M$ is not $\mu$ $\sigma$-finite. 
\end{defn}
Davies \cite{D} showed 
that there is a compact subset of $\R$ which is strongly invisible.
In a Polish group a {\bf Davies set} is a compact set which is strongly invisible. Many natural examples of Borel sets which are strongly invisible were given in \cite{MT}.

We would like to discuss visibility of a ``randomly" chosen compact or Cantor set. Unfortunately, even in the case of the reals, there is no suitable natural measure on the set 
of compact or Cantor sets. Hence, we use the notion of genericity. Let $X$ be a Polish space. A set $M \subseteq X$ is {\bf meager} it if is the countable union of nowhere dense sets. 
The set of meager sets forms a $\sigma$-ideal, i.e., subset of a meager sets is meager and the countable union of meager sets is meager. Moreover, as $X$ is complete, the Baire 
category theorem holds and hence no nonempty open set is meager. One thinks of meager sets as a collection of small sets and its complements as big sets. A set is {\bf comeager} 
if its complements is meager. We say that a {\bf generic element of $X$ has property $P$} means that the set of elements of $X$ which has property $P$ is comeager in $X$. 
 
\section{Visible sets} \label{sec-visible}
In this section we study visible sets and $\H$-visible sets. In particular,
Theorem~\ref{visibility} shows that a large class of Cantor sets are $\H$-visible.
Proposition~\ref{dec} shows that this includes the class of Cantor sets $K_{\alpha}$, $\alpha \in \DS$, studied in \cite {CMMS04}. Then, we discuss how big are the classes of visible, 
$\H$-visible and strongly invisible sets. We show that the class of $\H$-visible and strongly invisible sets are dense in $\C ([0,1])$ and, moreover, a generic compact subset of $[0,1]$ 
is visible. It remains open whether a generic compact subset of $[0,1]$ is $\H$-visible.\\
 
\subsection{$\H$-visibility}

We now introduce the construction necessary for our main $\H$-visibility theorem.\\

For $\sigma \in \N ^{<\N}$, $\sigma = \sigma_1 \dots \sigma_n$, we denote by $|\sigma|=n$ the length of $\sigma$, and for $k < |\sigma|$, $\sigma|k$ are the first $k$ digits of $\sigma$.  We say that $\sigma \in \N ^{<\N}$ is an extension of $\tau \in \N ^{<\N}$ if $|\sigma| > |\tau|$ and $\sigma||\tau| = \tau$. Further, if $\sigma \in \N^{\N}$, $\sigma|n$ is the restriction of $\sigma$ to the first $n$ digits.
 
A {\bf tree T} is simply a subset of $\N ^{<\N}$ with the property that if $\sigma \in T$ and $\sigma$ is an extension
of $\tau$, then $\tau \in T$. The {\bf body } of  $T$, denoted by $[T]$, is 
$$ [T] = \{ \sigma \in \N ^ {\N}: \sigma | n  \in T \mbox { for all } n \in \N \}.$$ For a tree $T$ and $\sigma \in T$, 
the valency of $\sigma$ in $T$ is 
the cardinality of $\{n \in \N: \sigma n \in T \}$. A tree $T$ is a {\bf Cantor tree} if for each $\sigma \in T$,  the valency of $\sigma$ 
is finite and at least 2.  If $T$ is a Cantor tree and each $\sigma \in T$ has valency $n$, then $T$ is an {\bf $n$-Cantor tree.}

Let $X$ be a Polish space and $T$ be a tree. A function $f$ from $T$ into the collection of all nonempty open subsets of $X$ is called a {\bf $T$-assignment }
into $X$ if the following conditions are satisfied.
\begin{enumerate}
\item For each $\sigma \in T$, $f(\sigma)$ is a nonempty open subset of $X$.
\item The diameter of $f(\sigma)$ is less than $1/|\sigma|$ for all $\sigma \in T$.
\item If $\sigma, \tau \in T$ with $\sigma \neq \tau$ and $|\sigma| = |\tau|$, then $f(\sigma) \cap f(\tau) = \emptyset$.
\item If $\sigma, \tau \in  T$ with $\tau$ an extension of $\sigma$, then, $\overline{f(\tau)} \subseteq f(\sigma)$.
\end{enumerate}

If $f$ is a $T$-assignment into $X$, then we let 
$$[f] =\bigcup _{\sigma \in [T]} \bigcap_{n=1}^{\infty} \overline{f({\sigma | n})}.$$

The following proposition is obvious. 

\begin{prop} Let $X$ be a Polish space, $T$ be a Cantor tree and $f$ be a $T$-assignment into $X$. Then, $[f]$ is 
a Cantor set. \end{prop}

For the following definitions assume that $T$ is a Cantor tree and $f$ is a $T$-assignment into $X$. We say
that  $f$  is a {\bf regular T-assignment} if for all $n \in \N$ the following holds: $$ \max \{ \diam (f(\sigma)): \sigma \in T , |\sigma| =n+1\} <
\min \{ diam (f(\sigma)): \sigma \in T , |\sigma| =n\}.$$ 

Let $ 2 \le l < \infty$. We say that $f$ satisfies the {\bf $l$-intersection condition} if the following condition holds.

For each $n \in \N$ and  each open ball $B$
in $X$, if $B$ intersects at least $l$ elements of $\{f (\sigma): \sigma \in T, |\sigma| =n\}$, then we must have that 
$\overline{f(\sigma)} \subseteq B$ for some $\sigma \in T$ with $|\sigma | =n$.

\begin{thm}\label{visibility} Let $X$ be a Polish space, $T$ be a $n$-Cantor tree, and $f$ be a $T$-assignment which is regular
and satisfies the $l$-intersection condition. Then, $[f]$ is $\H$-visible. \end{thm} 
\begin{proof} For each $k \in \N$, let 
\begin{align*}
m_k & = \min \{\text{diam}(f(\sigma)): \sigma \in T, |\sigma| =k\} \quad \mbox{and}\\ 
M_k & = \max \{\text{diam}(f(\sigma)): \sigma \in T, |\sigma| =k\}.
\end{align*}
  Let $h \in \H$ such that $h([m_k,M_k])= n^{-k}, k \in \N$.  This is possible
since $f$ is regular. We claim that $[f]$ is a $h$-set for $\mu^h$. 

It is clear from the construction that $\mu^h([f]) \le 1$. 

Let now $j_0 \in \N$ be such that $n^{j_0} > l$ and $c = n^{-j_0 -1}$. We will show that $\mu^h ([f])  \ge c$. 

For this, let
${\mathcal C}$ be any collection of open balls in $X$
which covers $[f]$. It suffices to show that $\sum_{B \in {\mathcal C}} h(\text{diam}(B)) \ge c$. Let $C_1, C_2, \ldots C_t$ be distinct elements
of ${\mathcal C}$ such that 
$$[f] \subseteq \bigcup_{i=1}^t C_i\quad \mbox{and each} \quad C_i \cap [f] \neq \emptyset.$$
 Let $\lambda$ be the {\em Lebesgue
number}  associated with the covering $C_1, C_2, \ldots C_t$ and $[f]$ such that if $B$ is any open set with $\text{diam} (B) < \lambda$ and $B \cap [f] \neq \emptyset$, then
$B \subseteq C_i$ for some $1 \le i \le t$. 

Let $k_0 \in \N$ be such that $1/k_{0} < \lambda$ and each of $C_i$, $1 \le i \le t$, contains
more than $n^{j_0}$ many elements from $\{f(\sigma): \sigma \in T,  | \sigma | = k_0\}$. 

Now for each $1 \le i \le t$, let $c_i$ be the cardinality of 
$$\{f (\sigma) : \sigma \in T, |\sigma| =k_0, \overline{f (\sigma)} \subseteq C_i \}$$
 and let $d_i \in \N$ be such that $n^{d_i} < c_i \le n^{d_i +1}$. We note that each
$d_i > j_0$ and 
$$\sum_{i=1}^t c_i \ge n^{k_0}\mbox{ as }1/k_0 < \lambda.$$ 

Fix $1 \le i \le t$. We next observe that $C_i$ intersects at least $l$ elements from the collection 
$$\{f(\sigma): \sigma \in T, |\sigma| = k_0- (d_i -j_0) \}.$$
For otherwise, we would have that $C_i$ intersects less than $l \cdot n^{d_i- j_0}$ elements from the collection $\{f(\sigma): \sigma \in T, |\sigma| = k_0 \}$.
As  $l \cdot n^{d_i- j_0} = l \cdot n^{-j_0} \cdot n^{d_i}< n^{d_i} <  c_i$, this would lead to a contradiction. 

Now since $f$ satisfies the $l$-intersection condition,
we have that $\overline{f(\sigma_i)} \subseteq C_i$ for some $\sigma_i \in T$ with $|\sigma_i| =k_0-(d_i -j_0)$.

Now, 
\begin{eqnarray*}
\sum_{i=1}^ t h (\text{diam}(C_i)) & \ge   & \sum_{i=1}^ t h( \text{diam}(f(\sigma_i)))\\
 & \ge & \sum_{i=1}^t n^{- \left ( k_0 - (d_i -j_0) \right ) }\\
 & = &  n^{-k_0} n^{-j_0}\sum_{i=1}^ t n ^ {d_i}\\
\end{eqnarray*}

Since $n^{d_i} < c_i \le n^{d_i +1}$ and $\sum_{i=1}^t c_i \ge n^{k_0}$, we have that $\sum_{i=1}^ t n ^ {d_i} > n^{k_0}/n$. Hence, we have
that $ \sum_{i=1}^ t h (\text{diam}(C_i)) > n^{-j_0}/n = n^{-j_0 -1} =c $, as claimed.
\end{proof}

We will now show that the Cantor sets $K_{\alpha}$, considered in \cite{CMMS04} fall under the hypothesis of the Theorem~\ref{visibility}, and therefore the latter theorem extends the result obtained in the earlier paper.
\begin{prop}\label{dec} Let $ \alpha \in \DS$.  There is a regular $T$-assignment $f$ from a $2$-Cantor tree, satisfying the $3$-intersection condition, such
that $[f] = K _{\alpha}$.  \end{prop}
\begin{proof} 
Recall that $K _{\alpha} = K_{\varphi}$ where $\varphi \in \G$ is defined as  $\varphi(1/2) =\alpha (1) $, $\varphi(1/4)=\alpha (2) $, 
$\varphi(3/4) =\alpha (3), \dots, \varphi(\frac{2s+1}{2^j})=\alpha (2^{j-1}+s) $.  In addition, let $\varphi(0) = 0$, $\Phi(0) = 0$, and $\Phi(1) =1$.

Let $T$ be a $2$-Cantor tree, i.e. $T = \bigcup_{k=1}^{\infty}\{0,1\}^k$.

To define $f$, let $d_{\sigma} = \sum_{i=1}^{|\sigma|} \frac{\sigma_i}{2^i}$ be the dyadic rational associated to $\sigma$.
For $\sigma \in \{0,1\}^{k}$, define
$$ f(\sigma) := \left(\Phi(d_{\sigma})+\varphi(d_{\sigma}), \Phi(d_{\sigma} + \frac{1}{2^{k}})\right).$$
 
 In this way,  $\{\overline{f (\sigma)}: \sigma \in T, |\sigma| =k\}$ is the union of $2^k$ consecutive disjoint intervals, hence $f$ clearly satisfies the $3$-intersection property. Further,
 $$K_{\alpha} = \bigcap_{k=1}^{\infty} \bigcup_{\{\sigma \in T, |\sigma| = k\}} \overline{f(\sigma)}=
 \bigcup_{\sigma \in [T]} \bigcap_{n=1}^{\infty} \overline{f(\sigma|n)}.$$
 In addition, since $\alpha$ is decreasing, $\diam(f(\sigma)), |\sigma| = k$ is at least the
diameter of $f(\sigma)$ for any $\sigma$ with $|\sigma| = k+1$, and therefore $f$ is a regular $T$-assignment.
\end{proof}

The previous theorem showed that there is in fact a very large class of Cantor sets that are $\H$-visible.
We will now state a Lemma, that will allow us to conclude, that in fact, the set of Cantor sets that are $\H$-visible, is dense in $\C(\R)$.

\begin{lem}\label{density}
Any collection ${\mathcal L}$ of compact subsets of $\R$
 satisfying the following properties is dense in the set of all compact subsets of $\R$.
 \begin{enumerate}
 \item ${\mathcal L}$ contains sets with arbitrarily small diameters. 
 \item If $A \in {\mathcal L}$ and $A_1, \ldots, A_n$ are translates of $A$,
 then $\cup_{i=1}^n A_i  \in {\mathcal L}$.
 \end{enumerate}
\end{lem}

\begin{cor} The collection of $\H$-visible compact subset of $\R$ is dense in $\C(R)$. \end{cor}
\begin{proof} We first recall that the standard middle third Cantor set is $\H$-visible just as any clopen subset of it. Further, any translation of any of these sets is also $\H$-visible. Using the Lemma, we have the result.
\end{proof}

\subsection{Generic Visibility}

We next show that a generic compact subset of $\R$ is visible.

We need some lemmas in order to prove the theorem.

\begin{lem}
If $B\subset \R$ is a set that is linearly independent over the rationals and $\alpha \in \R, \alpha \not= 0$, then
$(B+\alpha ) \cap B$ contains at most one point.
\end{lem}
\begin{proof} To obtain a contradiction, assume that $\alpha \neq 0$, $x_1 \neq x_2$
are points in $B$ such that $y_1= x_1 + \alpha, \ y_2= x_2+ \alpha \in B$. Clearly,
$x_1 \neq x_2$, $y_1 \neq y_2$, $y_1 \neq x_1$ and $y_2 \neq x_2$. If $x_1 = y_2$,
then we have that $\alpha = x_1 -x_2$, which leads to $y_1 = x_1 + x_1 -x_2$, contradicting that $B$ is linearly independent over the rationals. An analogous argument show that $x_2 \neq y_1$. Hence, we have that all of $x_1, x_2, y_1, y_2$ are distinct.
However, this implies that $y_1 -x_1 - y_2  + x_1 =0$, contradicting that $B$ is linearly independent over the rationals. 
\end{proof}
\begin{lem}\label{genericlem}
If $K$ is an uncountable, compact set that is linearly independent over the rationals, then there exists a translation invariant Borel measure $\mu$ such that $0< \mu(K) < +\infty$.
\end{lem}
\begin{proof} Let $\nu$ be any non-atomic Borel measure such that $\nu(K) = 1$.
Let $\mathcal B_K = \{ B \subset K: B \text{ is open relative to } K\}, $ and $\mathcal C := \{ B + t, \ B \in  \mathcal B_K, \ t \in \R\} \cup \{\R\}$
and let $P: \mathcal C \longrightarrow \R$ be a set function defined as:
\begin{align*}
P(\emptyset) & = 0,\\
P(B+t) & = \nu(B) \quad B \in  \mathcal B_K,  t \in \R, \mbox{ and }\\
P(\R) & = +\infty.
\end{align*}
$P$ is a premeasure and we use Method II to construct the sought after measure $\mu$:
$$\mu(A) = \lim_{\delta \rightarrow 0} \mu_{\delta}(A)\quad \text{where}\quad
\mu_{\delta} (A) = \text{inf} \{ \sum P(C_i) : \diam (C_i) \leq \delta, C_i \in \mathcal C, \cup_i C_i \supset A\}.$$

Since we use method II, we know that $\mu$ is Borel, and metric.
We need to show that:
\begin{itemize}
\item $\mu$ is translation invariant and
\item $0< \mu(K) <+\infty$.
\end{itemize}

The first part is a direct consequence of the definition of $P$.

For the second, it is clear that $\mu(K) \leq 1$.

For the other inequality, let $\{C_i\} \subseteq \C$ be a covering of $K$ with $\diam(C_i) \leq \delta, C_i = B_i + t_i$. Since $K$ is compact, 
$$
K \subseteq (B_1+\alpha_1 ) \cup (B_2 + \alpha_2) \cup \dots  \cup ( B_n+\alpha_n), \mbox{ where } \text (B_i+\alpha_i) \cap K  \not= \emptyset.
$$
By the linear independence over the rationals of $K$, if $\alpha_i \not= 0$, $(B_i+\alpha_i)  \cap K $ contains exactly one point and we call it $x_i$. Therefore, we can find $j_1, \ldots j_k$ for which
$C_{j_i} = B_{j_i}$ and 
$$ 
K \setminus \{x_{i_1}, \dots, x_{i_l}\} \subseteq B_{j_1} \cup B_{j_2} \cup \dots \cup B_{j_k}.
$$
Then
$$\sum_{s=1}^k P(B_{j_s}) = \sum_{s=1}^k\nu(B_{j_s}) \geq \nu (\cup_{s=1}^k B_{j_s})
$$
since $\nu$ is a Borel measure.
But
$$
\nu (\cup_{s=1}^k B_{j_s}) \geq  \nu (K \setminus \{x_{i_1}, \dots, x_{i_l}\}) = \nu(K) = 1.
$$
This yields the desired result.
\end{proof}
We are now ready for the main theorem of this section.
\begin{thm}\label{generic} A generic compact subset of $\R$ is visible. \end{thm}

\begin{proof} Now the proof follows from Lemma~\ref{genericlem} and the fact that a generic compact subset of $\R$ is uncountable and it follows from Theorem 19.1, \cite{K} that it is linearly independent over the rationals.
\end{proof}

\section{Davies sets in Polish abelian groups} \label{sec-invisible}
In this section we extract out key ideas from the construction of Davies \cite{D} to give a general procedure for 
constructing compact strongly invisible sets  in an abelian Polish group.

Let $G$ be an abelian Polish group and $\{t_k\}$ be a sequence of distinct points in $G$ such that $\sum_{k=1}^{\infty} d(0, t_k) < \infty.$
For each $\sigma \in \2n$ let $p(\sigma) = \sum_{k=1}^{\infty} \sigma(k) t_k$. 

We say that sequence {\bf $\{t_k\}$ is good} if 
$$ \left ( p(\sigma) = p(\tau),  \ \sigma , \tau \in \2n \right) \implies  \left (\sigma = \tau \right ),$$ i.e.,
each element of $p(\2n)$ has a unique expansion in terms of $\{t_k\}$.
 \begin{lem}\label{Lemma}
Let $G$ be an uncountable, abelian Polish group. Then $G$ has a good sequence.
\end{lem}
\begin{proof} Let $\{t_k\}$ be any sequence in $G$ such that for all $k \ge 1$ we have that $$\sum_{j=k+1}^{\infty} d(0,t_j) < \frac{1}{2} d(0,t_k).$$
Then, this $\{t_k\}$ has the desired property. \end{proof}
\begin{thm} Every uncountable abelian Polish group contains a strongly invisible compact  set.
\end{thm}
\begin{proof}
Let $G$ be an uncountable abelian Polish group, $\{t_k\}$ a good sequence in $G$ constructed as in Lemma \ref{Lemma} and as before for each $\sigma \in \2n$ let 
 $p(\sigma) = \sum_{k=1}^{\infty} \sigma(k) t_k$.
 
Morover, assume that $\{A_k: k=0,1,...\}$ is a decreasing sequence of subsets of $\N$ such that $A_0=\N$ and $A_k \setminus A_{k+1}$ is infinite for all 
$k= 0, 1,...$.

Finally,  define for $k\geq 1,$
 
  $$B_k =\{ \sigma \in \2n: \sigma(i) = 0   \text{ if }  i<k \text { or }  i \in A_k  \}.$$
and $$ C_k = \{ p(\sigma): \sigma \in B_k\}.$$

The properties of  the sequence  $\{t_k\}$ imply that $\sum_{k=1}^{\infty} \text{diam}(C_k) < \infty$.

We consider now a sequence $\{x_n\}$ of distinct points in  $G$ converging to zero, and a sequence of pairwise disjoint balls $\{B(x_n)\}$ centered at $x_n$.

Since the diameters of the sets $C_k$ go to zero, it is possible to find an increasing sequence $\{n_k\}$ of positive integers, such that
$C_{n_k} +l_k$  is included in the ball $B(x_k)$ for some translation $l_k \in G$.
Define  $C$ to be  the union $\bigcup_k (C_{n_k}+l_k)$ plus the element zero. So, $C$ is compact.

We will see that the set $C$ has the required properties.

For $k, l \in \N$,
denote by $B_{k,l}^0 = \left\{\sigma \in B_k: \sigma(k)=\dots =\sigma(k+l-1)=0\right\}$ and by $C_{k,l}^0=\{p(\sigma): \sigma \in B_{k,l}^0\}.$

Then $C_k$ is a  finite union of disjoint translates of $C_{k,l}^0$ since  
$$C_k= \bigcup_u \left( C_{k,l}^0 + u \right) $$ where 
$ u \in \left\{  \sum_{i=0}^{l-1}\alpha_i t_{k+i} :  \alpha_i \in \{0,1\}, 
\ \alpha_i = 0 \mbox{ if } k+i \in A_k \right\} $.

We want to see now that $C_{k+l}$ is the disjoint union of uncountable translates of $C_{k,l}^0$.
For this, define 
$$D_{k,l}=\{\tau \in \{0,1\}^N: \tau(s) = 0 \mbox{ if } s<k+l \mbox{ or } s \notin A_k \setminus A_{k+l}\}$$

The set $D_{k,l}$ is uncountable and so is the set $\{p(\tau): \tau \in D_{k,l}\}$ because the sequence $\{t_k\}$ is good.

It is easy to see now that the collection  of translates $\{ C_{k,l}^0 + p(\tau): \tau \in D_{k,l}\}$ is pairwise disjoint and that

$$C_{k+l}=\bigcup_{\tau \in D_{k,l}} \left( C_{k,l}^0 + p(\tau) \right) $$

Now we are ready to see that $C$ is strongly invisible.

Let $\mu$ a translation invariant Borel measure. If $\mu(C)$ is not zero, then there exists $k\in \N$ such that $\mu(C_{n_k}) >0$. Then $C_{n_k,n_{k+1}-n_k}^0$ has positive measure. 
It follows that $C_{n_{k+1}}$ is not $\sigma-$finite with respect to the measure $\mu$. So, neither  is $C$. 

\end{proof}

Using this construction, together with Lemma~\ref{density}, we can show the following proposition.
\begin{prop}
The set of strongly invisible sets is dense in $\C(\R)$.
\end{prop}

\begin{proof}
Following \cite{D} we
have that there are compact strongly invisible sets of arbitrary small diameters. Moreover, the finite union of translates of a fixed strongly invisible set  is again strongly invisible. Using  Lemma~\ref{density}, we have the desired result.
   \end{proof}

\end{document}